\documentclass[12pt,a4paper]{amsart}
\usepackage{graphicx,amssymb}
\usepackage{amsmath,amssymb,amscd}
\input xy
\xyoption{all}
\usepackage[all]{xy}
\usepackage{hyperref}

\newcommand{\ra}{\rightarrow}

\newcommand{\PP}{\mathbb P}

\newcommand{\cO}{\mathcal{O}}

\newcommand{\IM}{\mbox{im}}

\newcommand{\Hom}{\mbox{Hom}}

\theoremstyle{plain}
\newtheorem{theorem}{Theorem}[section]
\newtheorem{lem}[theorem]{Lemma}
\newtheorem{prop}[theorem]{Proposition}
\newtheorem{cor}[theorem]{Corollary}

\theoremstyle{definition}
\newtheorem{rem}[theorem]{Remark}
\newtheorem{ex}[theorem]{Example}
\numberwithin{equation}{section}
\begin{document}
\title[Clifford indices]{Vector bundles of rank 2 computing Clifford indices}

\author{H. Lange}
\author{P. E. Newstead}

\address{H. Lange\\Department Mathematik\\
              Universit\"at Erlangen-N\"urnberg\\
              Cauerstra\ss e $11$\\
              D-$91058$ Erlangen\\
              Germany}
              \email{lange@mi.uni-erlangen.de}
\address{P.E. Newstead\\Department of Mathematical Sciences\\
              University of Liverpool\\
              Peach Street, Liverpool L69 7ZL, UK}
\email{newstead@liv.ac.uk}

\thanks{Both authors are members of the research group VBAC (Vector Bundles on Algebraic Curves). The second author 
would like to thank the Department Mathematik der Universit\"at 
         Erlangen-N\"urnberg for its hospitality}

\begin{abstract}
Clifford indices of vector bundles on algebraic curves were introduced in a previous paper of the authors.
In this paper we study bundles of rank 2 which compute these Clifford indices. This is of particular interest 
in the light of recently discovered counterexamples to a conjecture of Mercat.
\end{abstract}
\maketitle

\section{Introduction}\label{intro}

Let $C$ be a smooth irreducible curve of genus $g\ge4$ defined over an algebraically closed field of characteristic zero. 
In describing the geometry of $C$, an important r\^ole is played by the classical {\em Clifford index} $\gamma_1$, 
which may be defined as follows. For any line bundle $L$ of degree $d_L$ on $C$, consider $\gamma(L):=d_L-2(h^0(L)-1)$. 
Then we define
$$
\gamma_1:=\min_L\{\gamma(L)\;|\;h^0(L)\ge 2,h^1(L)\ge2\}
$$
or equivalently
$$
\gamma_1:=\min_L\{\gamma(L)\;|\;h^0(L)\ge 2,d_L\le g-1\}.
$$
A line bundle $L$ satisfying the conditions of the first definition is said to {\em contribute} to $\gamma_1$; if in 
addition $\gamma(L)=\gamma_1$, then $L$ is said to {\em compute} $\gamma_1$.

In a previous paper \cite{cl}, we introduced Clifford indices for vector bundles, 
generalising the classical definition, as follows.
For any vector bundle $E$ of rank $r_E$ and degree $d_E$ on $C$, consider
$$
\gamma(E) := \frac{1}{r_E} (d_E - 2(h^0(E) -r_E)) = \mu(E) -2\frac{h^0(E)}{r_E} + 2.
$$
We then define:
$$
\gamma_n := \min_{E} \left\{ \gamma(E) \left|  
\begin{array}{c}   E \;\mbox{semistable of rank}\; n, \\
h^0(E) \geq n+1,\; \mu(E) \leq g-1
\end{array} \right\} \right.
$$
and
$$
\gamma_n' := \min_{E} \left\{ \gamma(E) \;\left| 
\begin{array}{c} E \;\mbox{semistable of rank}\; n, \\
h^0(E) \geq 2n,\; \mu(E) \leq g-1
\end{array} \right\}. \right.
$$
Note that $\gamma_1$ is just the classical Clifford index of the curve $C$ and $\gamma_1'=\gamma_1$. We say that $E$ 
{\it contributes to} 
$\gamma_n$ (respectively $\gamma_n'$) if $E$ is semistable of rank $n$ with $\mu(E) \leq g-1$ and $h^0(E) \geq n+1$ 
(respectively $h^0(E) \geq 2n$). If in addition 
$\gamma(E) = \gamma_n$ (respectively $\gamma(E) = \gamma_n'$), we say that $E$ {\it computes} $\gamma_n$ 
(respectively $\gamma_n'$).

Our object in the present paper is to study the bundles which compute $\gamma_2$ or $\gamma_2'$. In order to describe our results, we need some further definitions.
The {\em gonality sequence} 
$d_1,d_2,\ldots,d_r,\ldots$ of $C$ is defined by 
$$
d_r := \min \{ d_L \;|\; L \; \mbox{a line bundle on} \; C \; \mbox{with} \; h^0(L) \geq r +1\}.
$$
We have always $d_r<d_{r+1}$ and $d_{r+s}\le d_r+d_s$; in particular $d_n\le nd_1$ for all $n$ (see \cite[Section 4]{cl}). 
We say that $d_r$ {\em computes} $\gamma_1$ if $d_r \leq g-1$ and $d_r-2r=\gamma_1$ and that $C$ has {\em Clifford dimension} $r$ if $r$ is 
the smallest integer for which $d_r$ computes $\gamma_1$.

Following a section of preliminaries, we proceed in Section \ref{spc} to consider curves of Clifford dimension $2$, in other 
words smooth plane curves. In this case, we can describe all the bundles computing either $\gamma_2$ or $\gamma_2'$ 
(Propositions \ref{pr3.4}, \ref{pr3.5} and \ref{prop3.3}). In Section \ref{exc}, we consider curves of Clifford 
dimension $\ge3$; these are also known as {\em exceptional curves} (see \cite{elms}). We determine all the bundles 
which compute $\gamma_2$ (Theorem \ref{thm4.3} and Proposition \ref{prop4.4}). We are not able to determine all 
bundles computing $\gamma_2'$ except when $r=3$ (Proposition \ref{prop4.4}) or $r=4$ (Proposition \ref{prop4.6}), 
but we do describe how they arise when $r\ge5$ and $g=4r-2$ (it is conjectured that all exceptional curves have 
genus $4r-2$). 

An interesting by-product of this investigation is that, for $r\ge4$, we have $\gamma_2'<\gamma_1$, 
yielding further counterexamples to Mercat's conjecture in rank $2$ (see proposition \ref{propmer}) to add to those already described in \cite{fo} 
and \cite{ln2}. In particular Proposition \ref{prop4.6} and Theorem \ref{thm4.9} give the following theorem.

\begin{theorem} \label{thm1.1}
Let $C$ be a curve of Clifford dimension $r \geq 4$ of genus $g = 4r-2$.
Then there exists a stable bundle $E$ of rank $2$ and degree $\leq 4r-3$ on $C$ with $h^0(E) \geq 4$. In particular 
$$
\gamma'_2 < \gamma_1.
$$ 
\end{theorem} 

In Section \ref{htt}, we start the investigation of curves of Clifford dimension $1$ by looking at hyperelliptic, 
trigonal and tetragonal curves. The most general result that we obtain is 

\begin{theorem} \label{thm1.2}
Let $C$ be a general tetragonal curve of genus $g \geq 8$.
 Then the tetragonal line bundle $Q$ is unique and 
the only bundles computing $\gamma_2$ are
\begin{enumerate}
\item $Q \oplus Q$;
\item possibly non-trivial extensions
\begin{equation*} 
0 \ra Q \ra E \ra K \otimes Q^* \ra 0
\end{equation*}
with $h^0(E) = h^0(Q) + h^0(K \otimes Q^*) = g-1$. 
\end{enumerate}
In particular $\gamma_2 = \gamma'_2 = \gamma_1$. When $g \geq 27$, bundles of type $(2)$ do not exist.\\
{\em (See Theorem \ref{thm5.11} and Remarks \ref{rem5.12} and \ref{rem5.15}.)}
\end{theorem}

In Section \ref{kgonal} we study $k$-gonal curves for $k\ge5$ and prove a similar result to Theorem \ref{thm1.2} 
(Theorem \ref{thm6.2} and Corollary \ref{cor6.3}).

Section 7 concerns general curves. For such curves it is conjectured (see \cite{fo}) that $\gamma'_2 = \gamma_1$ and 
this is certainly true for $g \leq 16$. We work out the possible bundles computing $\gamma'_2$ under this assumption
(see Theorem \ref{prop7.5}).

In the final section we consider curves with $\gamma'_2 < \gamma_1$. Examples of such curves are known for all genera $g \geq 11$
(see \cite{fo}, \cite{ln2} and Theorem \ref{thm1.1} above). In this case we show that all bundles computing $\gamma'_2$ 
are stable with $h^0(E) = 2 + s, \; s \geq 2$, and do not possess a line subbundle with $h^0 \geq 2$ (we refer to such 
bundles as bundles of type PR). We show that $s \leq \gamma'_2 - \frac{\gamma_1}{2}$ (Proposition \ref{prop8.1}).
In one case we get an almost complete description of the bundles computing $\gamma'_2$.

\begin{theorem} {\em (Theorem 8.3)}
Suppose $\gamma'_2 < \gamma_1$ and $d_4 = 2 \gamma'_2 + 4$. Then the set of bundles 
of type PR with $s=2$ which compute $\gamma'_2$ is in bijective correspondence with the set of line bundles
$$
U(d_4,5) := \left\{ M \;{\Big |}\; \begin{array}{c}
                      d_M = d_4,\quad  h^0(M) = 5,\\
                      S^2H^0(M) \ra H^0(M^2) \; \mbox{not injective}
                      \end{array} \right\}.
$$
\end{theorem}

It is interesting to note that the condition that $S^2H^0(M) \ra H^0(M^2)$ be not injective can be restated in terms of
Koszul cohomology and that there are close connections between the problems discussed here and the maximal rank conjecture (see \cite{fo}). Our results also have implications for the non-emptiness of higher rank Brill-Noether loci, but we have not developed this here because we have no ``unexpected'' results for general curves.\\ 

We suppose throughout that $C$ is a smooth irreducible curve of genus $g\ge4$ defined over an algebraically closed field of characteristic zero and that $K$ denotes the canonical 
line bundle on $C$.

\section{Preliminaries}\label{prelim}

In this section, we recall a number of results from \cite{cl} and \cite{ln} and prove some additional lemmas.
First recall from \cite[Corollary 5.3, Lemma 2.2 and Theorem 5.2]{cl} that
\begin{equation} \label{eq4.1}
\gamma_2 = \min \left\{ \gamma_1, \frac{d_2}{2} - 1 \right\} \quad \mbox{and} \quad 
\gamma_1 \geq \gamma'_2 \geq \min \left\{ \gamma_1, \frac{d_4}{2} - 2 \right\}.
\end{equation} 
Also \cite[Lemma 4.6]{cl}
\begin{equation}\label{eq0}
d_r\ge\min\{\gamma_1+2r,g+r-1\}
\end{equation}

Next we have from \cite{ln}
\begin{lem}\label{lem1}
Any bundle computing $\gamma_2$ or $\gamma_2'$ is generated.
\end{lem}

We now recall that, for any generated line bundle $L$ with $h^0(L)=3$, one can define a vector bundle $E_L$ of rank $2$ by 
means of the evaluation sequence
\begin{equation}\label{eq2}
0\ra E_L^*\ra H^0(L)\otimes{\mathcal O_C}\ra L\ra 0.
\end{equation}

\begin{lem}\label{lem0} Let $E_L$ be defined by \eqref{eq2}. Then $h^0(E_L)\ge3$. Moreover, 
\begin{itemize}
\item[(i)] if $d_L\le2d_1$, then $E_L$ is semistable; 
\item[(ii)] if $d_L<2d_1$, then $E_L$ is stable;
\item[(iii)] if $d_L=d_2<2d_1$, then  $h^0(E_L)=3$.
\end{itemize}
\end{lem}
\begin{proof}
Dualising \eqref{eq2}, we see that $h^0(E_L)\ge3$ and $E_L$ is generated. Since also $h^0(E_L^*)=0$, it follows that any quotient line bundle of $E_L$ has $h^0\ge2$ and therefore has degree $\ge d_1$. This gives (i) and (ii).

(iii) is a special case of \cite[Theorem 4.15(a)]{cl}.
\end{proof}

\begin{lem}\label{lem2}
If $E$ is a semistable bundle of rank $2$ with $h^0(E)\ge3$, then $d_E\ge d_2$. Moreover, if $E$ computes $\gamma_2$ but not $\gamma_2'$, then $E\simeq E_L$ for some line bundle $L$ of degree $d_L=d_2$ with $h^0(L)=3$.\end{lem} 
\begin{proof}
The first statement is the case $n=2$ of \cite[Proposition 4.11]{cl}. If $E$ computes $\gamma_2$ but not $\gamma_2'$, then certainly $h^0(E)=3$; so $d_E\ge d_2$. Moreover $E$ is generated by Lemma \ref{lem1}, so we have an exact sequence
$$0\ra L^*\ra H^0(E)\otimes{\mathcal O}\ra E\ra 0,$$
where $L\simeq\det E$ is a generated line bundle of degree $d_L\ge d_2$ with $h^0(L)\ge3$. In order to minimise $\gamma(E)$, we must take $d_L=d_2$ and then $h^0(L)=3$ and $E\simeq E_L$.
\end{proof}

\begin{cor}\label{cor2}
If $d_2 \leq 2\gamma_1+2$, then the bundles computing $\gamma_2$ but not $\gamma_2'$ are precisely the bundles $E_L$ 
for $L$ a line bundle of degree $d_2$ with $h^0(L)=3$.
\end{cor}

\begin{proof}
By the lemma, any bundle computing $\gamma_2$ but not $\gamma_2'$ has the form $E_L$. Since $d_2\le2\gamma_1+2\le2d_1-2$, it follows from Lemma \ref{lem0} that $E_L$ is stable and $h^0(E_L)=3$. A direct computation using \eqref{eq4.1} gives $\gamma(E_L)=\gamma_2$.
\end{proof}

In discussing $\gamma_2'$, we shall make much use of the Lemma of Paranjape and Ramanan \cite[Lemma 3.9]{pr} (see also \cite[Lemma 4.8]{cl}), which we now state for the case of bundles of rank $2$.

\begin{lem}\label{lempr}
Let $E$ be a vector bundle of rank $2$ with $h^0(E)=2+s$ for some $s\ge1$. Suppose that $E$ has no line subbundle $M$ with $h^0(M)\ge2$. Then $h^0(\det E)\ge2s+1$ and, in particular, $d_E\ge d_{2s}$.
\end{lem}

As a complement to this lemma, we have
\begin{lem}\label{lempr2}
Suppose that $E$ is a semistable bundle of rank $2$ and degree $\le2g-2$ which possesses a subbundle $M$ with $h^0(M)\ge2$. Then $\gamma(E)\ge\gamma_1$, with equality if and only if $\gamma(M)=\gamma(E/M)=\gamma_1$ and all sections of $E/M$ lift to $E$.
\end{lem}
\begin{proof}
(This follows the proof of \cite[Theorem 5.2]{cl}.) By semistability, we have $d_M\le g-1$, so $M$ contributes to $\gamma_1$ and $\gamma(M)\ge\gamma_1$. Moreover $d_{E/M}\ge d_M$, so, if  $h^0(E/M)\le h^0(M)$, then $\gamma(E/M)\ge\gamma(M)\ge\gamma_1$. Note also that
$$d_{K\otimes (E/M)^*}=2g-2-d_E+d_M\ge d_M.$$
Hence, if $h^0(K\otimes (E/M)^*)\le h^0(M)$, then $\gamma(E/M)=\gamma(K\otimes(E/M)^*)\ge\gamma_1$. If neither of these possibilities occurs, then $E/M$ contributes to $\gamma_1$, so again $\gamma(E/M)\ge\gamma_1$. The result now follows from the fact that $\gamma(E)\ge\frac12(\gamma(M)+\gamma(E/M))$ with equality if and only if all sections of $E/M$ lift to $E$.
\end{proof}

In \cite{mer}, V. Mercat made a conjecture concerning the number of sections that a semistable bundle $E$ on a curve of given Clifford index may have. In the case of rank $2$, this conjecture can be expressed as follows.
\begin{itemize}
\item[(i)] If $\gamma_1+2\le\mu(E)\le g-1$, then $\gamma(E)\ge\gamma_1$;
\item[(ii)] if $\frac{\gamma_1+3}2\le\mu(E)<\gamma_1+2$, then $h^0(E)\le3$;
\item[(iii)] if $1\le\mu(E)<\frac{\gamma_1+3}2$, then $h^0(E)\le2$.
\end{itemize}
We observed in \cite[Proposition 3.3]{cl} that this conjecture implies that $\gamma_2'=\gamma_1$. In fact, we have
\begin{prop}\label{propmer}
Mercat's conjecture for rank $2$ holds if and only if $\gamma_2'=\gamma_1$.
\end{prop}
\begin{proof} By \cite[Proposition 3.3]{cl}, the conjecture implies that $\gamma_2'=\gamma_1$. Conversely, suppose that $\gamma_2'=\gamma_1$. Then certainly (i) holds. If $\mu(E)<\gamma_1+2$ and $h^0(E)\ge4$, then $E$ contributes to $\gamma_2'$. On the other hand $\gamma(E)<\gamma_1+2-2=\gamma_1$, a contradiction. Finally, suppose that $\mu(E)<\frac{\gamma_1+3}2$ and $h^0(E)\ge3$. Then $\gamma(E)<\frac{\gamma_1+3}2 -1=\frac{\gamma_1+1}2$; this contradicts \eqref{eq4.1}.
\end{proof}

Finally in this section we prove three lemmas which will be useful in determining when all sections of a quotient $E/M$ lift to $E$.

\begin{lem} \label{l2.3}
There exists a non-trivial extension of vector bundles
$$
0 \ra F \ra E \ra G \ra 0
$$
with the property that all sections of $G$ lift to $E$ if and only if the multiplication map
\begin{equation} \label{e2.2}
\alpha: H^0(G) \otimes H^0(K \otimes F^*) \ra H^0(K \otimes F^* \otimes G)
\end{equation}
is not surjective.
\end{lem}

\begin{proof}
All sections of $G$ lift to $E$ if and only if the extension class is in the kernel of the canonical map
$$
H^1(G^* \otimes F) \ra \Hom(H^0(G),H^1(F)).
$$
We require the condition that this kernel is non-trivial which is the case if and only if the dual map \eqref{e2.2}
is not surjective.
\end{proof}

\begin{lem}  \label{mult}
Suppose $G$ is a generated line bundle with $h^0(G) = 2$ and $F$ is a stable vector bundle with $\mu(F) = \deg G$.
Then we have
$$
h^0(F \otimes G) \geq \left\{ \begin{array}{ccc}
                              2 h^0(F)  & \mbox{if} & F \not \simeq G,\\
                              3 & \mbox{if} & F \simeq G.
                              \end{array} \right.
$$
Moreover,
$$
\mbox{\em codim im}(\alpha) = \left\{ \begin{array}{lcc}
                              h^0(F \otimes G) - 2h^0(F) & \mbox{if} & F \not \simeq G,\\
                              h^0(G \otimes G) - 3 & \mbox{if} & F \simeq G. 
                              \end{array}  \right.
$$
\end{lem}

\begin{proof}
Since $G$ is generated with $h^0(G) = 2$, the first assertion follows from the base-point-free pencil trick.
Note also that by the base-point-free pencil trick, 
$$
\ker(\alpha) \simeq H^0(G^* \otimes F^* \otimes K) \simeq H^1(F \otimes G)
$$
and hence
\begin{eqnarray*}
\dim \IM(\alpha) & = & 2[h^0(F) - d_F + r_F(g-1)] \\
&& - [h^0(F \otimes G) - d_F  - r_F d_G + r_F(g-1)]\\
& = & 2h^0(F) -h^0(F \otimes G) -d_F + r_F d_G + r_F (g-1).
\end{eqnarray*}
Since $F$ is stable and $\mu(F) = \deg G$, we have $h^1(G \otimes F^* \otimes K) = 0$ if $F \not \simeq G$ and $= 1$ 
if $F \simeq G$. So 
$$
h^0 (G \otimes F^* \otimes K) = \left\{ \begin{array}{lcc}
                                        r_F(g-1) + r_F d_G -d_F & \mbox{if} & F \not \simeq G,\\
                                        r_F(g-1) + r_F d_G -d_F + 1 & \mbox{if} & F \simeq G.
                                        \end{array} \right.
$$
The result follows.                                         
\end{proof}

For the final lemma we need a definition which we shall use several times. A curve is said to be a {\it Petri curve}
if the map 
$$
H^0(L) \otimes H^0(K \otimes L^*) \ra H^0(K)
$$
is injective for all line bundles $L$. It is important to note that the general curve of any genus is a Petri curve and that 
\begin{equation} \label{eqdr}
d_r = g + r - \left[ \frac{g}{r+1} \right].
\end{equation}
The equation \eqref{eqdr} is straightforward from the definitions.

\begin{lem} \label{lem2.10}
Let $C$ be a Petri curve of genus $g$ and $Q$ a line bundle of degree $d_1$ computing $\gamma_1$. Then

{\em (1)} $h^0(Q^2) = 3$ if $g$ is even;

{\em (2)} $h^0(Q^2) = 4$ if $g$ is odd.
\end{lem}

\begin{proof}
By Riemann-Roch,
\begin{eqnarray*}
h^0(K \otimes {Q^*}^2) & = & g-2d_1 -1 + h^0(Q^2)\\
& = & \left\{ \begin{array}{lcl}
              h^0(Q^2) - 3 & \mbox{if} & g \; \mbox{is even},\\
              h^0(Q^2) - 4 & \mbox{if} & g \; \mbox{is odd},
              \end{array}  \right.
\end{eqnarray*}
since $d_1 = \left[ \frac{g+3}{2} \right]$. We claim now that $h^0(K \otimes {Q^*}^2) =0$ on a Petri curve giving the result.

To prove the claim, suppose that $0 \neq s \in H^0(K \otimes {Q^*}^2)$.
Consider the commutative diagram
$$ 
\xymatrix{
H^0(Q) \otimes H^0(Q) \ar[r]  \ar[d]_{id \otimes ( \cdot s)} & H^0(Q^2) \ar[d]^{\cdot s}\\
H^0(Q) \otimes H^0(K \otimes Q^*) \ar[r] & H^0(K).
    }
$$
The left hand vertical homomorphism is clearly injective and for a Petri curve the bottom map as well.
Since $h^0(Q) > 1$, the top horizontal map is not injective, a contradiction.
\end{proof}

\section{Smooth plane curves}\label{spc}

Let $C$ be a curve of Clifford dimension $2$, in other words a smooth plane curve of degree $\delta\ge5$. We recall (see, for example, \cite[Section 8]{cl}) that
\begin{equation} \label{eq3.1}
\gamma_1 = \delta - 4, \quad d_1 = \delta-1, \quad d_2 = \delta, \quad d_3=2\delta-2, \quad d_4 = 2 \delta -1.
\end{equation}
Moreover, the hyperplane bundle $H$ is the unique line bundle of degree $\delta$ with $h^0(H) =3$ and $H$ and $K\otimes H^*\simeq H^{\delta-4}$ are the only line bundles computing $\gamma_1$.
It is well known that $C$ is projectively normal in $\PP^2$. So all multiplication maps 
$H^0(H^r) \otimes H^0(H^s) \ra H^0(H^{r+s})$ are surjective.

\begin{prop} \label{pr3.4}
For a smooth plane quintic the only bundle computing $\gamma_2 = \gamma'_2 = 1$ is
$H \oplus H.$
\end{prop}

\begin{proof} By \eqref{eq4.1} and \eqref{eq3.1}, we see that $\gamma_2 = \gamma'_2 = 1$. 
Let $E$ be a bundle computing $\gamma_2$. By Lemma \ref{lem2}, we have $d_E \geq 5$. Moreover, 
$h^0(E) = \frac{d_E}{2} + 1$, implying that $h^0(E)\ge4$ and $E$ computes $\gamma_2'$.
If $E$ has no line subbundle $L$ with $h^0(L) \geq 2$, then, writing $h^0(E) = 2 + s$ with $s \geq 2$, 
it follows from Lemma \ref{lempr} that $d_E \geq d_{2s} \geq d_4 + 2s -4$.
So $\gamma(E) \geq \frac{d_4}{2} - 2 = \delta - \frac{5}{2} > 1$, a contradiction.

So $E$ must occur in an exact sequence 
\begin{equation} \label{eq3.2}
0 \ra M \ra E \ra N \ra 0
\end{equation}
with $h^0(M) \geq 2$. By Lemma \ref{lempr2}, we must have $\gamma(M)=\gamma(N)=1$ and hence $M \simeq N \simeq H$. 

We need to show that the extension \eqref{eq3.2} must be trivial.
By Lemma \ref{l2.3} it suffices to show that the map $H^0(H) \otimes H^0(K \otimes H^*) \ra H^0(K)$ is 
surjective, which is the case, since $K \simeq H^2$.
\end{proof}

\begin{prop} \label{pr3.5}
For a smooth plane sextic the bundles computing $\gamma_2 = 2$ are
$E_H$ and $H \oplus H$.
The only bundle computing $\gamma'_2 = 2$ is $H \oplus H$.
\end{prop}

\begin{proof}
The fact that $\gamma_2=\gamma_2'=2$ follows from \eqref{eq4.1} and \eqref{eq3.1}. By Corollary \ref{cor2} and \eqref{eq3.1}, $E_H$ is the only bundle computing $\gamma_2$ but not $\gamma_2'$.

Suppose now that $E$ computes $\gamma_2'$ and write $h^0(E)=2+s$, $s\ge2$. If $E$ has no line subbundle with $h^0 \geq 2$, then 
Lemma \ref{lempr} implies that $d_E\ge d_{2s}$. Moreover
$$d_{2s}\ge d_4+2s-4>2s+4$$
by \eqref{eq3.1}. So $\gamma(E)>2$, a contradiction. Thus $E$ occurs in an exact sequence \eqref{eq3.2} with $h^0(M) \geq 2$; moreover, by Lemma \ref{lempr2}, both $M$ and $N$ compute $\gamma_1$.

Noting that $H$ and $H^2$ are the only line bundles computing $\gamma_1$, we have either $M\simeq N\simeq H$ or $M\simeq H$, $N\simeq H^2$. Since all sections of $N$ must lift to $E$ and $H^0(N)\otimes H^0(K\otimes M^*)\to H^0(K\otimes M^*\otimes N)$ is surjective in both cases, this allows only the split extension. Since $E$ is semistable, only $E\simeq H\oplus H$ is possible. \end{proof}

\begin{prop} \label{prop3.3}
Let $C$ be a smooth plane curve of degree $\delta \geq 7$. 
Then $\gamma_2 < \gamma'_2 = \gamma_1$ and 

{\em (i)} $E_H$ is the only bundle computing $\gamma_2$; moreover,
$E_H$ is stable with $h^0(E_H) = 3$;

{\em (ii)} $H \oplus H$ is the only bundle computing $\gamma'_2$.
\end{prop}

\begin{proof} 
It follows from \eqref{eq4.1} and \eqref{eq3.1} that 
$$
\gamma_2 = \frac{\delta}{2} -1 \quad \mbox{and} \quad \gamma'_2 = \gamma_1 = \delta -4.
$$
Hence $\gamma_2 < \gamma'_2$. (i) now follows from Corollary \ref{cor2} and Lemma \ref{lem0}.

(ii) Suppose that $E$ computes $\gamma'_2$. If $h^0(E) = 2 + s$ with $s \geq 2$ and $E$ has no line subbundle $M$ with 
$h^0(M) \geq 2$, then by Lemma \ref{lempr},
$$
\gamma(E) \geq \frac{d_{2s}}{2} - s \geq \frac{d_4}{2} - 2 = \delta - \frac{5}{2} > \gamma_1,
$$
a contradiction. It follows from Lemma \ref{lempr2} that $E$ can be written as an extension
\eqref{eq3.2} with $h^0(M) \geq 2$ and $\gamma(M)=\gamma(N)=\gamma_1$; moreover all sections of $N$ lift to $E$. The proof is now completed exactly as for Proposition \ref{pr3.5}, noting that in this case either $N\simeq H$ or $N\simeq H^{\delta-4}$.
\end{proof}

\section{Exceptional curves}\label{exc}

In this section we consider curves of Clifford dimension $\ge3$, in other words curves for which neither $d_1$ nor $d_2$ computes $\gamma_1$.
\begin{lem} \label{lem4.1}
Let $C$ be a curve of Clifford dimension  $r \geq 3$. Then
$$
d_1 > \frac{d_2}{2}, \quad \frac{d_2}{2} > \frac{d_3}{3} \quad \mbox{and} \quad \gamma_1 \geq \frac{d_2}{2} -1.
$$
The last inequality is strict for $r \geq 4$.
\end{lem} 

\begin{proof}
According to \cite[Corollary 3.5]{elms}, $d_r \geq 4r-3$.
Since $d_1$ does not compute $\gamma_1$, we have
$$
d_1 \geq \gamma_1 + 3 = d_r -2r + 3.
$$
Since $d_2 \leq d_r - r + 2$, this gives the first inequality.
Since $d_2$ does not compute $\gamma_1$, we have 
$$
d_2 \geq \gamma_1 + 5 = d_r -2r + 5.
$$
Using the fact that $d_3 \leq d_r - r + 3$, we obtain the second inequality. For the third inequality, we have
$$
d_2 \leq d_r -r + 2 \leq 2d_r -4r + 2 = 2 \gamma_1 + 2,
$$
since $d_r \geq 4r-3 \geq 3r$. By the same reason this inequality is strict for $r \geq 4$.
\end{proof}

\begin{theorem} \label{thm4.3}
Let $C$ be of Clifford dimension $r \geq 3$. Then there is a semistable bundle $E$ of rank $2$ and
degree $d_2$ with $h^0(E) = 3$. Any such bundle computes $\gamma_2$, is of the form $E \simeq E_L$ and is stable.

If $r \geq 4$ and $d_4 > d_2 + 2$, then these are the only bundles computing $\gamma_2$. 
\end{theorem}

\begin{proof}
Taking account of Lemma \ref{lem4.1}, the first part follows from Lemma \ref{lem0} and \eqref{eq4.1}.

If $r \geq 4$, we have $\gamma_1 > \frac{d_2}{2} -1$ by Lemma \ref{lem4.1}. If $d_4 > d_2 + 2$, 
then $\frac{d_4}{2} - 2 > \frac{d_2}{2}-1$.
So $\gamma'_2 > \gamma_2$ by \eqref{eq4.1}. Hence, by Corollary \ref{cor2}, these bundles of degree $d_2$ with $h^0 = 3$ are 
the only bundles computing $\gamma_2$. 
\end{proof}

Let $C$ be a curve of Clifford dimension 3. Then $C$ is a complete intersection of 2 cubics in $\PP^3$ \cite{m}.
The curve $C$ is of genus 10 with
$$
\gamma_1 = 3, \; d_1 = 6,\; d_2 = 8, \; d_3 = 9 \;\; \mbox{and} \;\; d_4 = 12.
$$
Let $H$ denote the hyperplane bundle on $C$. 

\begin{prop} \label{prop4.4}
Let $C$ be a curve of Clifford dimension $3$. Then $\gamma_2 = \gamma'_2 = \gamma_1 = 3$ and the only bundles computing $\gamma_2$ 
are the bundles $E_L$ of Theorem \ref{thm4.3} and $H \oplus H$. The only bundle computing $\gamma'_2$ is $H \oplus H$.
\end{prop} 

\begin{proof}
First, by \eqref{eq4.1} we have $\gamma_2 = \gamma'_2 = \gamma_1$.
By Corollary \ref{cor2}, the bundles computing $\gamma_2$ but not $\gamma'_2$ are precisely the bundles $E_L$. Now suppose
that $E$ computes $\gamma'_2$. If $E$ has no line subbundle with $h^0 \geq 2$, then, by
Lemma \ref{lempr}, $d_E \geq d_{2s}$ with $s \geq 2$ and hence
$$
\gamma(E) \geq \frac{d_{2s}}{2} - s \geq \frac{d_4}{2} - 2 = 4,
$$
a contradiction. 

So, by Lemma \ref{lempr2}, $E$ fits into an extension \eqref{eq3.2} with $\gamma(M) = \gamma(N) = 3$. The only possibility is $M \simeq N \simeq H$ and all sections of $N$ must lift to $E$.
It follows from the projective normality of $C$ in $\PP^3$ and Lemma \ref{l2.3} that $E\simeq H \oplus H$. 
\end{proof}

Now let $C$ be a curve of Clifford dimension $r = 4$. Then $C$ has genus 14 and has a semicanonical projectively 
normal embedding into $\PP^4$ given by the unique line bundle $H$ of degree 13 computing $\gamma_1$. In particular
we have
$$
\gamma_1 = 5, \; d_1 = 8 \; \mbox{and} \; d_4 = 13.
$$

\begin{lem}  \label{lem4.5}
$d_2 = 10, \; d_3 = 12$ and $d_6 = 18$.
\end{lem}

\begin{proof}
From Macdonald's secant plane formula (see \cite[p. 351]{acgh}) we see that there exist trisecants 
of $C$ in $\PP^4$. Projecting from any such trisecant, we obtain a line bundle of degree 10 with $h^0 \geq 3$.
Since no line bundle of degree $< 13$ can compute $\gamma_1$, there does not exist a line bundle of degree 9 
with $h^0 \geq 3$ or of degree 11 with $h^0 \geq 4$. Hence $d_2 = 10$ and $d_3 = 12$. The last assertion follows by duality.
\end{proof}

\begin{prop} \label{prop4.6}
Let $C$ be a curve of Clifford dimension $4$. Then

{\em (i)} $\gamma_2 = 4$ and is computed by a bundle of the form $E_L$ with $L$ a line bundle of degree $10$ with $h^0(L) =3$
and by no other bundles;

{\em (ii)} $\gamma'_2 = \frac{9}{2}$ and is computed by a unique bundle $E$ of degree $13$ with $\det E = H$. 
\end{prop}

\begin{proof}
(i) Since $d_4 > d_2 + 2$, this is included in Theorem \ref{thm4.3}.

(ii) According to \cite{elms} $C$ is contained in a K3-surface $X$ which is embedded by a complete linear 
system into $\PP^4$. Hence $X$ is the complete intersection of a quadric and a cubic in $\PP^4$.
It follows from \cite[Remark 3.4]{gmn} and the uniqueness of $H$ that there exists a unique stable rank-2 
bundle $E$ of degree 13 with $h^0(E) = 4$ and that $\det E = H$. Certainly $\gamma(E) = \frac{9}{2}$.
So by \eqref{eq4.1},  $\gamma'_2 = \frac{9}{2}$ and is computed by $E$.

Let $E$ be any bundle of degree $> 13$ computing $\gamma'_2$. Write $h^0(E) = 2 + s$ with $s \geq 3$ and $d_E = 9 + 2s$.
If $E$ has a subbundle $M$ with $h^0(M) \geq 2$, then, by Lemma \ref{lempr2}, 
$\gamma(E) \geq \gamma_1$, a contradiction. So Lemma \ref{lempr} applies to give  
$$
d_E \geq d_{2s}.
$$
This fails for $s = 3$, since $d_6 = 18$, and hence for all $s \geq 3$.
\end{proof}

\begin{rem}
Since $\gamma'_2 < \gamma_1$, this gives a new example of a curve for which Mercat's conjecture fails.
\end{rem}

Now suppose that $C$ is of Clifford dimension $r \geq 5$ and of genus $g = 4r-2$. This implies that
\begin{equation}  \label{e4.2}
d_r = 4r-3, \;\; \gamma_1 = 2r-3, \;\; d_1 = 2r \;\;\mbox{and} \;\; d_{r-1} = 4r-4.
\end{equation}
For the first equation note that $d_r \leq g-1$ by the definition of $\gamma_1$ and $d_r \geq 4r-3$ by \cite[Corollary 3.5]{elms}.
The last 2 equalities come from the fact that the Clifford dimension of $C$ is $r$. According to \cite{elms}, 
there exist 
curves of this type for any $r$. Let $H$ denote the line bundle computing $\gamma_1$. Then $H$ gives a non-degenerate 
embedding of $C$ into $\PP^r$ (note that $C$ in $\PP^r$ must be smooth, since otherwise projection from a singular 
point would give Clifford dimension $< r$). 

Now consider the canonical map
$$
S^2H^0(H) \ra H^0(H^2).
$$
We have $\dim S^2H^0(H) = \frac{1}{2}(r+1)(r+2)$ and $h^0(H^2) \leq 4r-2$. It follows that $C$ is contained in at least
$$
\frac{1}{2}(r+1)(r+2) - 4r + 2= \frac{1}{2}(r^2 - 5r + 6) = { r-2 \choose 2}
$$
independent quadrics.

Therefore a dimensional computation shows that $C$ is contained in a quadric of rank $\leq 5$. In fact, $C$ cannot 
be contained in a quadric of rank $\leq 4$, since otherwise the systems of maximal linear subspaces on the quadric 
would give pencils $L_1, L_2$ on $C$ such that $H \simeq L_1 \otimes L_2$. This is impossible since $d_r < 2d_1$.
So $C$ lies on a quadric $q$ of rank 5. 
As such, $q$ contains a 3-dimensional system of $(r-3)$-planes  all of which contain the vertex of $q$.

If $C$ does not meet the vertex of $q$, then through each point of $C$ there is only a 1-dimensional system of $(r-3)$-planes. 
So there exists an $(r-3)$-plane not meeting $C$.

If $C$ does meet the vertex, it does so in at most a finite number of points. It follows that by projection we can obtain 
a non-degenerate morphism $C \ra \PP^s$ with $s < r$ such that the image of $C$ in $\PP^s$ is contained in a quadric $q'$ of
rank 5 and does not meet the vertex of $q'$. So there exists an $(s-3)$-plane on $q'$ not meeting $C$.

\begin{lem} \label{lem4.8}
There exists a non-degenerate morphism $C \ra \PP^s$ for some $s \leq r$ such that, if $H'$ denotes the hyperplane bundle 
of $C$ in $\PP^s$, then there exists a 3-dimensional subspace $W$ of $H^0(H')$ such that 
\begin{itemize}
\item $W$ generates $H'$;
\item the kernel $N$ of the linear map $W \otimes H^0(H') \ra H^0(H'^2)$ has dimension $\geq 4$.
\end{itemize}
\end{lem}

\begin{proof}
From the previous discussion we obtain the morphism $C \ra \PP^s$ such that $C$ in $\PP^s$ is contained in a quadric $q'$ of rank 5 
and there exists an $(s-3)$-plane $\Pi$ on $q'$ not meeting $C$. 

Now let $W$ be the 3-dimensional subspace of $H^0(H')$  which annihilates the $(s-2)$-dimensional subspace of $H^0(H')^*$ 
defined by $\Pi$.
Since $\Pi$ does not meet $C$, it follows that $W$ generates $H'$. Moreover, since $\Pi$ lies on $q'$, the image of 
$W \otimes H^0(H')$ in $S^2H^0(H')$ contains the 1-dimensional subspace corresponding to the quadric $q'$. It follows that
the image of $N$ in $S^2H^0(H')$ has dimension at least 1. Since the kernel of the map $N \ra S^2H^0(H')$ is $\bigwedge^2 W$, 
it follows that $\dim N \geq 4$.
\end{proof}

\begin{theorem} \label{thm4.9}
Let $C$ be a curve of Clifford dimension $r \geq 5$ of genus $g = 4r-2$.
Then there exists a stable bundle $E$ of rank $2$ and degree $\leq 4r-3$ on $C$ with $h^0(E) \geq 4$. In particular 
$$
\gamma'_2 \leq \gamma(E) < \gamma_1.
$$ 
\end{theorem}

\begin{proof}
Let $H'$ and $W$ be as in Lemma \ref{lem4.8}.
Define the bundle $E^*$ by the exact sequence
\begin{equation}  \label{eq4.3}
0 \ra E^* \ra W \otimes \cO_C \ra H' \ra 0.
\end{equation}
Tensoring by $H'$ and noting that $H' \simeq \det E$ and hence $E^* \otimes H' \simeq E$, we obtain
$$
0 \ra E \ra W \otimes H' \ra H'^2 \ra 0.
$$     
This implies $H^0(E) \simeq N$ and so $h^0(E) \geq 4$. Note that $d_E = d_{H'} \leq 4r-3$.

From \eqref{eq4.3} we get that $h^0(E^*) = 0$ and $E$ is generated. Hence any quotient line bundle $L$ of $E$ has $h^0(L) \geq 2$.
So $d_L \geq d_1 = 2r$ and $E$ is stable. Note that 
$$
\gamma(E) \leq \frac{4r-3}{2} -2 = 2r - \frac{7}{2} < \gamma_1.
$$
\end{proof}

\begin{cor} \label{cor4.10}
Let $C$ be a curve of Clifford dimension $r=5$. Then
$$
6 \leq \gamma'_2 \leq \frac{13}{2}.
$$
Moreover, if a bundle $E$ computes $\gamma'_2$, then $h^0(E) = 4$.
\end{cor}

\begin{proof}
According to \cite{elms}, $C$ is of genus 18. The bundle $E$ constructed in the theorem has $\gamma(E) \leq \frac{13}{2}$.
According to \eqref{e4.2}, $d_4 = 16$. So by \eqref{eq4.1}, $\gamma'_2 \geq 6$. 

Let $F$ be a bundle computing $\gamma'_2$.
If $h^0(F) = 2 + s$ with $s \geq 3$, then Lemmas \ref{lempr} and \ref{lempr2} imply that $d_F \geq d_{2s}$. Since 
$\gamma'_2 \leq \frac{13}{2}$, we have $d_F \leq 13 + 2s$. This contradicts the fact that $d_6 \geq 20$, which follows by
duality from the fact that $d_4 = 16$. 
\end{proof}
\begin{rem}\label{rmk4.10}
A curve $C$ of Clifford dimension $5$ has $\gamma_2'=6$ if and only if there exists a line bundle $M$ of degree $d_4=16$ with $h^0(M)=5$ and $S^2H^0(M)\to H^0(M^2)$ non-injective. This follows from \cite[Theorem 3.2]{gmn} and Proposition \ref{prop8.5} below.
From our previous discussion, the curve $C$ in $\PP^5$ lies on a number of quadric cones; if it passes through the vertex of one of these cones, then projection from this vertex gives the required bundle $M$. We do not know whether this can happen.
\end{rem}

\section{Hyperelliptic, trigonal and tetragonal curves}\label{htt}

From now on, we consider curves of Clifford dimension $1$. These are also known as $k$-gonal curves, 
where $k=d_1=\gamma_1+2$. In this section, we study the cases $2\le k\le4$, in other words 
hyperelliptic, trigonal and tetragonal curves. For these curves, it follows from \eqref{eq4.1} that
$$
\gamma_2 = \gamma'_2 = \gamma_1.
$$

For hyperelliptic curves it is well known that the only bundles computing $\gamma_2 = \gamma'_2 =0$ are the bundles 
$$
H^r \oplus H^r
$$
where $H$ is the hyperelliptic line bundle and $1 \leq r \leq \frac{g-1}{2}$ \cite[Proposition 2]{re}. 

\begin{prop} \label{pr5.1}
Let $C$ be a trigonal curve of genus $g \geq 5$ and denote by $T$ the trigonal line bundle. Then the bundles computing 
$\gamma_2$ are
\begin{itemize}
\item $T \oplus T$;
\item possibly stable bundles fitting into a non-trivial extension
\begin{equation} \label{eq5.1}
0 \ra T \ra E \ra K \otimes T^* \ra 0
\end{equation}
with $h^0(E) = h^0(T) + h^0(K \otimes T^*) =g$.
\end{itemize}
In particular every bundle computing $\gamma_2$ also computes $\gamma'_2$.
\end{prop}

\begin{proof}
If $E$ computes $\gamma_2$, then $d_E \geq d_2$ by Lemma \ref{lem2}. Hence, since $\gamma_2=1$, 
$$
h^0(E) = \frac{d_E}{2} + 1 \quad \mbox{and} \quad  6 \leq d_E \leq 2g-2
$$
with $d_E$ even. 

If $E$ has no line subbundle $M$ with $h^0(M) \geq 2$, then, by Lemma \ref{lempr},
$d_E \geq d_{r}$ with $r = 2(\frac{d_E}{2} -1) = d_E -2$. A simple numerical calculation using \eqref{eq0} gives a 
contradiction.

So we have an exact sequence
\begin{equation} \label{eq5.8}
0 \ra M \ra E \ra N \ra 0
\end{equation}
with $h^0(M) \geq 2$. Since $d_M \leq g-1$ by semistability, we have (see \cite[Remark 4.5(b)]{cl})
$$
h^0(M) \leq \frac{d_M}{3} + 1.
$$

If $d_N \leq g-1$, then $h^0(N) \leq \frac{d_N}{3} + 1$. So $h^0(E) \leq h^0(M) + h^0(N) \leq \frac{d_E}{3} +2$.
This contradicts $h^0(E) = \frac{d_E}{2} + 1$ except when $d_E =6$ and in this case $M\simeq N\simeq T$.
If $d_E =6$ and \eqref{eq5.8} does not split, then according to Lemma \ref{mult}
we have $h^0(T^2) \geq 4$, which is impossible since $d_{T^2} = 6$ and $d_3 \geq 7$.

If $d_N > g-1$, then 
\begin{eqnarray*}
h^0(N) &=& h^1(N) + d_N -g + 1\\
& \leq & \frac{d_{K \otimes N^*}}{3} + d_N - g + 2 =\frac{2d_N -g + 4}{3}.
\end{eqnarray*}
So $h^0(E) \leq \frac{d_M + 2d_N -g + 4}{3} + 1$, i.e. 
$$
\frac{d_M + d_N}{2} = \frac{d_E}{2} \leq \frac{d_M + 2d_N -g + 4}{3}
$$
which is equivalent to
$$
d_N -d_M \geq 2g -8.
$$
This requires $d_M = 3, \; d_N = 2g-5$ with $h^0(M) = h^0(K \otimes N^*) = 2$. Hence $M \simeq K \otimes N^* \simeq T$ 
and \eqref{eq5.8} becomes \eqref{eq5.1}.

If $E$ is not stable, then $E$ occurs in an extension 
$$
0 \ra M' \ra E \ra N' \ra 0
$$
with $\frac{d_E}{2} \leq d_{M'} < 2g-5$. The same calculation to estimate $h^0(E)$ gives $h^0(E) < \frac{d_E}{2} + 1$,
a contradiction.
\end{proof}

\begin{rem} \label{rem5.2}
Extensions \eqref{eq5.1} exist if and only if the multiplication map
\begin{equation} \label{eq5.3}
H^0(K \otimes T^*) \otimes H^0(K \otimes T^*) \ra H^0(K^2 \otimes {T^*}^2) 
\end{equation}
is not surjective.
For $g \geq 17$ these extensions do not exist.
(The condition $g \geq 17$ is probably not best possible.)
\end{rem}

\begin{proof}

Condition \eqref{eq5.3} follows from Lemma \ref{l2.3}.
The last assertion follows from \cite[Theorem 2(a)]{gl}. For this we have to show that $K \otimes T^*$ is very ample.
In fact for any $p,q \in C$, 
$$
h^0(K \otimes T^*(-p-q)) = h^1(T(p+q)) = h^1(T) -2 = h^0(K \otimes T^*) -2,
$$
since $h^0(T(p+q)) = h^0(T) = 2$. 
\end{proof}

\begin{rem} \label{rem5.3}
Let $M$ and $N$ be the trigonal line bundles on a general curve of genus 4. Then we have $M \otimes N\simeq K$.
A modified version of the previous argument shows that the only semistable bundles computing $\gamma_2$ are 
$$
M \oplus M, \quad N \oplus N \quad \mbox{and} \quad M \oplus N.
$$
For the trigonal curves of genus 4 for which there exists only one trigonal line bundle $M$ we have also a unique
semistable bundle $E$ with $h^0(E) = 4$ which occurs as a non-trivial extension $0 \ra M \ra E \ra M \ra 0$. 
\end{rem}

\begin{lem} \label{lem5.4}
Let $C$ be a tetragonal curve. Then $d_2 \geq 6$ and there exists a bundle $E$ computing $\gamma_2$ with $h^0(E) = 3$ if
and only if $d_2 = 6$.
\end{lem}

\begin{proof}
We cannot have $d_2 \leq 5$, since otherwise $\gamma_1 \leq 5 - 4 = 1$. A semistable bundle $E$ with $h^0(E) \geq 3$ has 
$d_E \geq d_2$ by Lemma \ref{lem2}. If $h^0(E) = 3$, then $\gamma(E) \geq \frac{1}{2}(d_2 -2)$. So $E$ can compute 
$\gamma_2 = 2$ only if $d_E = d_2 = 6$. 
For the existence of $E$ we take $E = E_L$ where $L$ is a line bundle of degree $d_2$ with $h^0(L) = 3$. 
\end{proof}

\begin{rem} \label{rem5.5}
The condition $d_2 = 6$ is satisfied when $g=5$ or $6$ and for all bielliptic curves.
\end{rem}

\begin{lem} \label{lem5.6}
Let $C$ be a tetragonal curve. If there exists a bundle $E$ computing $\gamma'_2$ with no line subbundle $M$ with 
$h^0(M) \geq 2$, then $g=5$ and $E$ fits into an exact sequence
\begin{equation} \label{e5.4}
0 \ra M \ra E \ra K \otimes M^* \ra 0
\end{equation}
with $d_M = 2$ and $h^0(M) =1$. Moreover, when $g = 5$, there exist non-trivial extensions \eqref{e5.4} 
for which all sections of $K \otimes M^*$ lift and all such $E$ are semistable.
\end{lem}

\begin{proof}
Write $h^0(E) = 2 + s$ with $s \geq 2$. By Lemma \ref{lempr}, $d_E\geq d_{2s} \geq d_4 + 2s -4$. So
$$ 
\gamma(E) \geq \frac{1}{2}(d_4 + 2s -4 -2s) = \frac{d_4}{2} -2.
$$
This implies $d_4 \leq 8$. A line bundle $L$ of degree $d_4$ with $h^0(L) = 5$ has $h^1(L) \geq 2$ if $d_4 \leq g + 2$
and therefore contributes to the Clifford index. This gives a contradiction for $g \geq 6$.

For $g = 5$, we must have $d_4 = 8, \;d_E = 8$ and $s=2$.
Moreover, by Lemma \ref{lempr}, we have $h^0(\det E) \geq 5$. So $\det E \simeq K$.
 By \cite{ms} $E$ has a line subbundle $M$ of degree $\geq 2$ 
with $h^0(M) \leq 1$. So $h^0(E/M) \geq 3$, which implies $d_{E/M} \geq d_2 = 6$. 
Since $E/M \simeq K \otimes M^*$, we obtain \eqref{e5.4} as required.

According to Lemma \ref{l2.3} there exist non-trivial extensions \eqref{e5.4} for which all sections of $K \otimes M^*$ 
lift if and only if 
$$
H^0(K \otimes M^*) \otimes H^0(K \otimes M^*) \ra H^0(K^2 \otimes M^{*2})
$$
is not surjective. The map factors through $S^2H^0(K \otimes M^*)$ which has dimension 6. On the other hand, 
$h^0(K^2 \otimes M^{*2}) = 8$. 

If $E$ is not semistable, then $E$ possesses a line subbundle $M'$ of degree 5. 
Now $h^0(M') \leq 2$ and $h^0(E/M') \leq 1$. So $h^0(E) \leq 3$, a contradiction.
\end{proof}

We now consider tetragonal curves of genus 5 starting with curves which satisfy the Petri condition.

\begin{prop} \label{prop5.7}
Let $C$ be a Petri curve of genus $5$. Then the bundles computing $\gamma_2$ are
\begin{enumerate} 
\item $E_L$ with $d_L = d_2 = 6$;
\item $Q \oplus Q'$, where $Q$ and $Q'$ are tetragonal line bundles; 
\item for each tetragonal line bundle $Q$ a unique bundle $E$ which is a 
non-trivial extension 
\begin{equation} \label{eqn5.5}
0 \ra Q \ra E \ra Q \ra 0
\end{equation}with $h^0(E)=4$:
\item  for each tetragonal line bundle $Q$ a unique bundle $E$ which is a 
non-trivial extension 
$$
0 \ra Q \ra E \ra K \otimes Q^* \ra 0
$$with $h^0(E)=4$;
\item possibly stable bundles fitting into a non-trivial extension \eqref{e5.4}.
\end{enumerate}
All these bundles compute $\gamma'_2$ except for those of type $(1)$. 
\end{prop}

\begin{proof}
The bundles $E_L$ and $Q \oplus Q'$ have $\gamma(E) = \gamma_2$ by direct numerical calculations.

If $E$ computes $\gamma_2$ and possesses a line subbundle $M$ with $h^0(M) \geq 2$, then by Lemma \ref{lempr2} there is a non-trivial extension $0 \ra M \ra E \ra N \ra 0$ with tetragonal line bundles $M = Q$ and $N = Q'$. By Lemma \ref{l2.3}, there exist such extensions for which all sections of $Q'$ lift if and only if the map
\begin{equation} \label{e5.5}
H^0(Q') \otimes H^0(K \otimes Q^*) \ra H^0(K \otimes Q^* \otimes Q')
\end{equation}
is not surjective. By Lemmas \ref{mult} and \ref{lem2.10} this happens if and only if either $Q' \simeq Q$ 
or $Q' \simeq K \otimes Q^*$.

By Lemma \ref{lem2.10}, the Petri condition implies that $Q \not \simeq K \otimes Q^*$. This gives the cases $(3)$ and $(4)$.
The uniqueness statement follows from Lemma \ref{mult}. 

If $E$ does not possess a line subbundle $M$ with $h^0(M) \geq 2$, we obtain an extension \eqref{e5.4} by Lemma \ref{lem5.6}.
Any strictly semistable bundles which are included among these are already contained in $(3)$ or $(4)$.
\end{proof}

\begin{rem}
It is not clear whether the bundles in $(5)$ actually exist. 
\end{rem}

\begin{rem}
If $C$ is not Petri and $Q^2 \simeq K$, then there 
exists a family of bundles $E$ parametized by $\PP^1$ which occur as non-trivial extensions \eqref{eqn5.5}.
Otherwise the results are the same as in Proposition \ref{prop5.7}. 
\end{rem}

\begin{prop} \label{prop5.8}
For a tetragonal curve of genus $g=6$ the bundles computing $\gamma_2$ are 
\begin{enumerate}
\item $E_L$ with $d_L = d_2 = 6$;
\item $Q \oplus Q'$, where $Q$ and $Q'$ are tetragonal line bundles;
\item if $C$ is not a Petri curve, bundles $E$ of degree $8$ with $h^0(E)=4$ given by non-trivial extensions
$$
0 \ra Q \ra E \ra Q \ra 0,
$$
where $Q$ is a tetragonal line bundle and $h^0(Q^2)=4$;
\item stable bundles $E$ of degree $10$ with $h^0(E) = 5$ given by non-trivial extensions 
$$
0 \ra Q \ra E \ra K\otimes Q'^* \ra 0
$$
where $Q$ and $Q'$ are tetragonal line bundles (such bundles exist when $Q\simeq Q'$).
\end{enumerate}
All these bundles compute $\gamma'_2$ except for those of type $(1)$. 
\end{prop}

\begin{proof}
The bundles $E_L$ and $Q \oplus Q'$ certainly 
have $\gamma(E) = \gamma_2=2$.

In view of Lemmas \ref{lem5.6} and \ref{lempr2}, all other bundles $E$ computing $\gamma_2$ arise as extensions
$$
0 \ra Q \ra E \ra N \ra 0
$$
where $Q$ is a tetragonal line bundle and $N$ is either a tetragonal line bundle $Q'$ or the Serre dual $K\otimes Q'^*$ of a tetragonal line bundle $Q'$. 
In both cases all sections of $N$
must lift to $E$. 

It is easy to check that $d_4=9$. Hence, when $N\simeq Q'$, it follows from Lemma \ref{mult} that, if the extension is non-trivial, then 
$E$ can exist only if $Q' \simeq Q $ and $h^0(Q^2)=4$. This 
cannot happen on a Petri curve by Lemma \ref{lem2.10}. This gives $(3)$.

The only remaining case is $(4)$. Here the extension must be non-trivial, since $E$ is semistable. Conversely, all non-trivial extensions of this type yield semistable bundles $E$ and all those with $h^0(E)=5$ are in fact stable. Existence when $Q\simeq Q'$ follows from Lemma \ref{l2.3} and the fact that $S^2H^0(K\otimes Q^*)$ has dimension $6$, while $h^0(K^2\otimes Q^{*2})=7$.
\end{proof}

\begin{prop} \label{prop5.9}
Let $C$ be a tetragonal curve of genus $g = 7$ such that $d_2 = 7$. Then the bundles computing $\gamma_2$ are
\begin{enumerate}
\item $Q \oplus Q'$, where $Q$ and $Q'$ are tetragonal line bundles;
\item possibly non-trivial extensions
\begin{equation*} \label{eq5.5}
0 \ra Q \ra E \ra Q \ra 0,
\end{equation*}
where $Q$ is a tetragonal line bundle with $h^0(E) =4$ (such extensions exist if and only if $h^0(Q^2) = 4$);
\item possibly non-trivial extensions
\begin{equation*} \label{eq5.6}
0 \ra Q \ra E \ra K \otimes Q'^* \ra 0
\end{equation*}
with $h^0(E) = h^0(Q) + h^0(K \otimes Q'^*) = 6$. 
\end{enumerate}
In particular, every bundle computing $\gamma_2$ also computes $\gamma'_2$.
\end{prop}

\begin{proof} 
By Lemma \ref{lem5.4} every bundle $E$ computing 
$\gamma_2$ has $h^0(E) \geq 4$. Hence by Lemma \ref{lem5.6}, $E$ possesses a line subbundle $M$ with $h^0(M) \geq 2$.
Writing $N = E/M$, it follows by Lemma \ref{lempr2} that $\gamma(M) = \gamma(N) = \gamma_1 = 2$
and all sections of $N$ lift to $E$.
Since $d_M \leq \mu(E) \leq g-1$, $M$ must be a tetragonal line bundle $Q$.
If $d_N \leq g-1$, the same holds for $N$; if $d_N > g-1$, then $N$ must be the Serre dual $K \otimes Q'^*$
of a tetragonal line bundle $Q'$.

Since $d_3 = 8$ by Serre duality, we have $h^0(Q \otimes Q') \le 4$. Lemma \ref{mult} now implies (2).
\end{proof}

\begin{theorem} \label{thm5.11}
Let $C$ be a tetragonal curve of genus $g \geq 8$ such that the tetragonal line bundles are the only 
line bundles of degree $\leq g-1$ which compute $\gamma_1$. Then the tetragonal line bundle $Q$ is unique and 
the only bundles computing $\gamma_2$ are
\begin{enumerate}
\item $Q \oplus Q$;
\item possibly non-trivial extensions
\begin{equation*} 
0 \ra Q \ra E \ra K \otimes Q^* \ra 0
\end{equation*}
with $h^0(E) = h^0(Q) + h^0(K \otimes Q^*) = g-1$. 
\end{enumerate}
In particular, every bundle computing $\gamma_2$ also computes $\gamma'_2$. 
\end{theorem}

\begin{proof}
If $g \geq 9$, then $d_3 \geq 9$ by hypothesis. If $g = 8$, the same holds by Serre duality. Hence, if $Q$ and $Q'$
are tetragonal line bundles, we have $h^0(Q  \otimes Q') = 3$. So by the base-point-free pencil trick, 
$Q \simeq Q'$. 
Moreover, by Lemma \ref{mult} there does not exist a non-trivial extension $0 \ra Q \ra E \ra Q \ra 0$
for which all sections of $Q$ lift to $E$.

Since $d_2\ge7$, the bundles $E_L$ with $d_L=d_2$ and $h^0(L)=3$ do not compute $\gamma_2$.
The proof now proceeds in the same way as the proof of the previous proposition.
\end{proof}

\begin{rem} \label{rem5.12}
The hypothesis is satisfied for a general tetragonal curve (see \cite[Remark 4.5(c)]{cl}). 
\end{rem}

\begin{ex}
The normalisation of a plane curve of degree $6$ with 2 nodes is a tetragonal curve of genus 8 with $d_2 = 6$ and $d_3 = 8$.
It possesses 2 tetragonal line bundles $Q$ and $Q'$; the hyperplane bundle $H$ also computes $\gamma_1$. The 2 tetragonal 
line bundles $Q$ and $Q'$ are non-isomorphic and $h^0(Q^2) = h^0(Q'^2) = 3$.
The bundle $E_H$ computes $\gamma_2$ but not $\gamma'_2$.
Moreover, the only bundles computing $\gamma'_2$ are
$$
Q \oplus Q, \quad Q \oplus Q', \quad Q' \oplus Q', \quad H \oplus H
$$
and unique non-trivial extensions
$$
0 \ra Q \ra E \ra K \otimes Q \ra 0 \quad \mbox{and} \quad 0 \ra Q' \ra E \ra K \otimes Q' \ra 0.
$$ 
To see this, we have to show according to Lemma \ref{l2.3} that the map
$H^0(K \otimes Q'^*) \otimes H^0(K \otimes Q^*) \ra H^0(K^2 \otimes Q^* \otimes Q'^*)$
is surjective and that the map $H^0(K \otimes Q^*) \otimes H^0(K \otimes Q^*) \ra H^0(K^2 \otimes {Q^*}^2)$ 
and the analogous map for $Q$ replaced by $Q'$ have 1-dimensional cokernel. 

The linear series $|K \otimes Q^*|$ is cut by conics through the node corresponding to $Q'$ with a similar statement for  
$|K \otimes Q'^*|$. Moreover, $|K^2 \otimes Q^* \otimes Q'^*|$ is cut by quartics through both nodes. It is easy to see that 
this gives the surjectivity. On the other hand, the linear series $|K^2 \otimes {Q^*}^2|$ has dimension 12, while the 
linear system of quartics with a double point at the corresponding node is only 11-dimensional. 

In the same way one checks that there are no non-trivial extensions $0 \ra H \ra E \ra H \ra 0$ or $0\to H\to E\to K\otimes H^*\to 0$ such that $E$ 
computes $\gamma'_2$.
\end{ex}

\begin{ex}
A smooth $(4,4)$-curve on a smooth quadric surface 
is a tetragonal curve of genus 9 with $d_2 = 7, \; d_3 = 8$ and possessing 2 tetragonal line bundles $Q$ and $Q'$. 
These bundles are non-isomorphic and $h^0(Q^2) = h^0(Q'^2) = 3$. The hyperplane bundle $H$ also computes $\gamma_1$.
The bundles computing $\gamma'_2$ are exactly as 
in the previous example.

This follows from the fact that $|K \otimes Q^*|$ and $|K \otimes Q'^*|$ are 
cut respectively by $(1,2)$- and $(2,1)$-curves 
on the quadric while $|K^2 \otimes Q^* \otimes Q'^*|$ is cut by $(3,3)$-curves. On the other hand,
the codimension of the linear system of $(4,2)$-curves in $|K^2 \otimes {Q^*}^2|$ is again 1. Similarly $|H|$ is cut by $(1,1)$-curves and $|H^2|$ by $(2,2)$-curves.
\end{ex}

\begin{rem} \label{rem5.15}
Under the hypotheses of Theorem \ref{thm5.11}, extensions of type (2) exist if and only if the multiplication map
$$
H^0(K \otimes Q^*) \otimes H^0(K \otimes Q^*) \ra H^0(K^2 \otimes {Q^*}^2)
$$
is not surjective. For a general tetragonal curve of genus $g \geq 27$ there are no such bundles.
(The condition $g \geq 27$ 
is probably not best possible.)
\end{rem}

\begin{proof}
The statement concerning existence is Lemma \ref{l2.3}.

The surjectivity of the multiplication map follows from \cite[Theorem 2(a)]{gl}. For this we have to show that 
$K \otimes Q^*$ is very ample. The proof is the same as in the proof of Remark \ref{rem5.2}. 
\end{proof}

Finally in this section, we look at tetragonal curves which are in some sense at the opposite extreme from the general ones, namely bielliptic curves. For this case we list some bundles computing $\gamma_2$ and $\gamma_2'$, but we do not know whether the list is complete.

\begin{prop} \label{prop5.11}
Let $\pi: C \ra C'$ be a double covering of an elliptic curve for which $C$ has genus $g$ and $\gamma_1 = 2$. Then, 
if $E'$ is a semistable rank-$2$ bundle on $C'$ of degree $d', \; 3 \leq d' \leq g-1$, the pull-back $\pi^*E'$ 
computes $\gamma_2$. For $d' \geq 4$ it also computes $\gamma'_2$. 
\end{prop}

\begin{proof}
The pull-back $\pi^*E'$ is semistable of degree $2d'$ and $h^0(\pi^*E') \geq d'$. So 
$$
\gamma(\pi^*E') \leq \frac{1}{2}(2d' -2(d'-2)) = 2.
$$
Since $\gamma_2 = 2$, we must have equality. For $\gamma'_2$ we need also $d' \geq 4$.
\end{proof}

\section{$k$-gonal curves for $k\ge5$}\label{kgonal}

\begin{prop} \label{prop6.1}
Suppose $d_2 = 2d_1$. Then $\gamma_2 = \gamma'_2 = \gamma_1$ and all bundles computing $\gamma_2$ compute $\gamma'_2$.
Moreover, there exists a unique semistable bundle $E$ of degree $d_2$ computing $\gamma'_2$ and
$$
E \simeq Q \oplus Q,
$$
where $Q$ is the unique line bundle of degree $d_1$ computing $\gamma_1$.
\end{prop}

\begin{proof} 
If $d_2 = 2d_1$, then $\gamma_1 \leq d_1 - 2 < \frac{d_2}{2}-1$. So $\gamma_2 = \gamma_1$ 
and hence also $\gamma'_2 = \gamma_1$.

If $h^0(E) = 3$, then $d_E \geq d_2$ by Lemma \ref{lem2}. This gives $\gamma(E) > \gamma_1$, a contradiction.

Let $E$ be a semistable bundle of rank 2 and degree $d_2$ with $h^0(E) \geq 4$.
Since $d_4 > d_2$, Lemma \ref{lempr} implies that $E$ has a line subbundle $M$ with $h^0(M) \geq 2$ and 
hence $d_M \geq d_1$. Lemma \ref{lempr2} now gives
an extension
\begin{equation} \label{eq6.1}
0 \ra M \ra E \ra N \ra 0
\end{equation}
with $M$ and $N$ line bundles of degree $d_1$ computing $\gamma_1$ and all sections of $N$ lift to $E$. 
Since $d_{M \otimes N} = d_2$, we have $h^0(M \otimes N) \leq 3$ which implies that $M \simeq N$ and the extension 
\eqref{eq6.1} splits by Lemma \ref{mult}. So $M \simeq Q$ is unique (provided it exists) and $E \simeq Q \oplus Q$. Clearly
$\gamma(Q \oplus Q) = \gamma'_2$.

For existence of $Q$, we need to know that $d_1$ computes $\gamma_1$. This holds
because $d_2=2d_1$ cannot hold on a curve of Clifford dimension $\ge2$. This is obvious for smooth plane curves; for exceptional curves, we have (see the proof of Lemma \ref{lem4.1})
\begin{equation}\label{eqnewnew}
2d_1-d_2\ge 2d_r-4r+6-(d_r-r+2)=d_r-3r+4\ge r+1.
\end{equation}
\end{proof}

\begin{theorem} \label{thm6.2}
Suppose $d_2 = 2d_1$ and the line bundle $Q$ of degree $d_1$ is the only line bundle of degree $\leq g-1$ 
computing $\gamma_1$. Then the bundles computing $\gamma_2 = \gamma'_2$ are
\begin{enumerate}
\item $Q \oplus Q$;
\item possibly non-trivial extensions
\begin{equation*} 
0 \ra Q \ra E \ra K \otimes Q^* \ra 0
\end{equation*}
with $h^0(E) = h^0(Q) + h^0(K \otimes Q^*) = g + 3 - d_1$. 
\end{enumerate}
\end{theorem}

\begin{proof}
Let $E$ be a bundle computing $\gamma'_2 = \gamma_2$ and write $h^0(E) = 2 + s,\; s \geq 2$. If $E$ has no 
line subbundle with $h^0 \geq 2$, then by Lemma \ref{lempr}, $d_E \geq d_{2s}$. So
$$
\gamma(E) \geq \frac{d_{2s}}{2} - s \geq \frac{d_2}{2} - 1 = d_1 - 1 > \gamma_1,
$$
a contradiction.
So by Lemma \ref{lempr2} there exists an extension 
$$
0 \ra M \ra E \ra N \ra 0
$$ with $M$ and $N$ line bundles computing $\gamma_1$ and such that any section of $N$ lifts to $E$.
The only possibilities are $M \simeq N \simeq Q$ and $M \simeq Q, \; N \simeq K \otimes Q^*$.
The rest is contained in Proposition \ref{prop6.1}.
\end{proof}

\begin{cor} \label{cor6.3}
For $k \geq 5$, let $C$ be a general $k$-gonal curve of genus $g > \max\{3k^2-8k+7,46\}$. Then the only bundle computing 
$\gamma_2 = \gamma'_2$ is $Q \oplus Q$, where $Q$ is the unique line bundle of degree $\leq g-1$ computing $\gamma_1$. 
\end{cor}

\begin{proof}
We use the fact (see \cite[Theorem 3.1]{k}) that
$$
d_r = kr \quad \mbox{for} \quad 1 \leq r \leq \frac{1}{k-2} \left[ \frac{g-4}{2} \right].
$$
In particular we have $d_2 = 2d_1$ for $g \geq 4k-4$ which holds by hypothesis. Now write
$$
r_0 := \left[ \frac{1}{k-2} \left[ \frac{g-4}{2} \right] \right].
$$
A line bundle $L$ with $d_L > kr_0$ has $h^0(L) - 1 \leq d_L -kr_0 + r_0$. Given that $d_L \leq g-1$, this implies that
$$
\gamma(L) \geq g-1 - 2(g-1 -kr_0 + r_0) = 2(k-1)r_0 -g + 1.
$$ 
To get $\gamma(L) > \gamma_1$ we therefore require
\begin{equation} \label{eq6.2}
2(k-1)r_0 -g + 1 > k-2.
\end{equation}
To prove this, note that $r_0 \geq \frac{\frac{g-5}{2} -k + 3}{k-2}$ which is equivalent to
$$
2(k-1)r_0 \geq g + \frac{1}{k-2} [g - (k-1)(2k-1)].
$$
It is therefore sufficient to have
$$
\frac{1}{k-2}[g - (k-1)(2k-1)] > k-3
$$
which is true by our hypothesis. Proposition \ref{prop6.1} implies the uniqueness 
of the line bundle computing $\gamma_1$. 

It remains to show that case (2) of Theorem \ref{thm6.2} does not occur. By Lemma 2.8 this means that we must show that the map
$$
H^0(K \otimes Q^*) \otimes H^0(K \otimes Q^*) \ra H^0(K^2 \otimes {Q^*}^2)
$$ 
is surjective. The argument of Remark \ref{rem5.2} shows that $K \otimes Q^*$ is very ample. Now by 
\cite[Theorem 2(a)]{gl} the map is surjective for 
\begin{equation}\label{eqsurj}
g > \max\left\{ \frac{k(k+1)}{2}, 10k-4 \right\}.
\end{equation}
This inequality holds under our hypothesis on $g$. 
\end{proof}

\begin{rem}The number $46$ as a lower bound for $g$ is required only to ensure that \eqref{eqsurj} holds when $k=5$.
\end{rem}

\begin{rem} \label{rem6.3}
Suppose $d_2 = 2d_1 -1$ and $d_3 > 2d_1$. Then $d_1$ computes $\gamma_1$ (see \eqref{eqnewnew}), $\gamma_2 = \gamma'_2 = \gamma_1$ and all bundles
computing $\gamma_2$ also compute $\gamma'_2$. Moreover, the proof of Proposition \ref{prop6.1} shows that there is a unique 
line bundle $Q$ of degree $d_1$ computing $\gamma_1$ and that $h^0(Q^2) = 3$. 
If $Q$ is the only line bundle of degree $\leq g-1$ computing $\gamma_1$, then the proof of Theorem \ref{thm6.2} works 
with the same conclusion.
\end{rem}

\begin{rem}
If $d_2 = 2d_1 -2$, then $d_1$ computes  $\gamma_1 = \frac{d_2}{2} - 1$. So again $\gamma_2 = \gamma'_2 = \gamma_1$. 
In this case the bundles $E_L$ with $L$ a line bundle of degree $d_2$ with $h^0(L) = 3$ compute $\gamma_2$, 
but not $\gamma'_2$.
\end{rem}

\begin{prop} \label{prop7.1}
If $d_1$ computes $\gamma_1$ and $d_2 \leq 2d_1 -2$, the bundles computing $\gamma_2$ but not $\gamma'_2$ are precisely the bundles $E_L$, 
where $L$ is a line bundle of degree $d_2$ with $h^0(L) = 3$, and all such bundles are stable.

If $d_1$ computes $\gamma_1$, $d_2 < 2d_1 -2$ and $d_4 > d_2+2$, then these are the only bundles computing $\gamma_2 < \gamma'_2$.
\end{prop}

\begin{proof}
Since $\gamma_1 = d_1 -2$, the first statement follows from Corollary \ref{cor2} and Lemma \ref{lem0}.

Now suppose $d_2 < 2d_1 -2$ and $d_4 > d_2+2$; then $\frac{d_4}{2}-2 > \frac{d_2}{2} -1$. So $\gamma'_2 > \gamma_2$ by
\eqref{eq4.1}. Hence the bundles $E_L$ are the only bundles computing $\gamma_2$. 
\end{proof}

\begin{rem} 
If $d_2 < 2d_1 -2$ and $d_4 = d_2 + 2$, then the bundles $E_L$ for $L$ of degree $d_2$ 
still compute $\gamma_2$. However there may exist further bundles computing simultaneously $\gamma_2$ and $\gamma'_2$.
If such bundles exist, then $\gamma'_2 < \gamma_1$. So they would give counterexamples to Mercat's conjecture.
\end{rem}

\section{General curves} \label{general}

For $g \leq 6$, Remark \ref{rem5.3} and Propositions \ref{prop5.7} and \ref{prop5.8} apply to general curves. 
In this section we consider general curves of genus $g \geq 7$.

\begin{prop} \label{prop7.3}
For a general curve of genus $g \geq 7, \; g \neq 8$, the bundles $E_L$ with $d_L = d_2$ compute $\gamma_2$ 
and are the only bundles computing $\gamma_2$.
\end{prop}

\begin{proof}
Recall that for a general curve, 
\begin{equation} \label{dr}
\gamma_1 = \left[\frac{g-1}{2} \right] \quad \mbox{and} \quad d_r = r + g - \left[ \frac{g}{r+1} \right]
\end{equation}
(see \cite[Remark 4.4(c)]{cl} and \eqref{eqdr}).
A simple numerical computation 
using \eqref{dr} shows that $d_2 < 2d_1 -2$ and $d_4 > d_2 + 2$. So Proposition \ref{prop7.1} applies.
\end{proof}

The general curve of genus 8 requires separate treatment because $d_2 = 2d_1 -2$.

\begin{prop} \label{prop7.4}
For a general curve of genus $g = 8$ the bundles computing $\gamma_2= \gamma'_2=3$ are 
\begin{enumerate}
\item $E_L$ with $d_L = d_2 = 8$;
\item $Q \oplus Q'$ with $d_Q = d_{Q'} = d_1 = 5$ and $h^0(Q) = h^0(Q') =2$;
\item stable bundles $E$ of degree $14$ with $h^0(E) = 6$ given by a non-trivial extension 
\begin{equation*} 
0 \ra Q \ra E \ra K \otimes Q'^* \ra 0
\end{equation*} 
(such bundles exist when $Q \simeq Q'$).
\end{enumerate} 
 
\end{prop}

\begin{proof}
The values of $d_r$ come from \eqref{dr}. It follows from \eqref{eq4.1} that $\gamma_2 = \gamma'_2 = \gamma_1$.  
The bundles $E_L$ and $Q \oplus Q'$ certainly 
have $\gamma(E) = \gamma_2$. 

We have also $h^0(Q^2) = 3$ by Lemma \ref{lem2.10} and it follows 
from Lemma \ref{mult} that there are no non-trivial extensions $E$ of $Q$ by $Q$ with $h^0(E) = 4$.

If $Q \not \simeq Q'$, then $h^0(Q \otimes Q') = 4$ since $d_4 = 11$. Again by Lemma \ref{mult} 
there are no non-trivial extensions $E$ of $Q$ by $Q'$ with $h^0(E) = 4$.

It remains to determine whether there exist any semistable rank-2 bundles $E$ of degree $d_E$ with 
$2d_1 < d_E \leq 2g-2$ and $\gamma(E) = \gamma_2$.

There are now 2 possibilities $d_E = 12, \; h^0(E) = 5$ and $d_E = 14$,  $h^0(E) = 6$.
Lemmas \ref{lempr} and \ref{lempr2} imply that $E$ must occur in an extension $0 \ra Q \ra E \ra N \ra 0$
with $\gamma(N) = \gamma_1$ 
and $h^0(N) = h^0(E) -2$. If $d_E = 12$, we have $h^0(N) = 3$ giving $d_N \geq d_2 = 8$, a contradiction. 
If $d_E = 14$, then $d_N = 9$ and $h^0(N) = 4$. So $N$ is the Serre dual of a line bundle $Q'$.

For the last statement we have to show according to Lemma \ref{l2.3} that
the map 
$$
H^0(K \otimes Q^*) \otimes H^0(K \otimes Q^*) \ra H^0(K^2 \otimes {Q^*}^2)
$$
is not surjective. In this case the map factors through $S^2H^0(K \otimes Q^*)$.  
However we have $\dim S^2H^0(K \otimes Q^*) = 10$ 
and $h^0(K^2 \otimes {Q^*}^2) = 11$. 

It remains to prove stability of $E$. If $E$ is not stable, it would have a line subbundle of degree $7$ or $8$. Since $d_2=8$, this implies that $h^0(E)\le5$, a contradicton.
\end{proof}

We now consider the problem of finding bundles computing $\gamma'_2$ for general curves of genus 
$g \geq 7,\; g \neq 8$. It follows from \eqref{eq4.1} and \eqref{dr} that $\gamma'_2 = \gamma_1$ 
if $C$ is a general curve of genus $\leq 10$. 
This has also been proved for $g \leq 16$ in \cite[Theorem 1.7]{fo} (it is a consequence of \eqref{eq4.1} and \eqref{dr}
for $g \leq 10,\; g = 12$ and $g=14$). It is conjectured in \cite{fo} that this holds
for general curves of arbitrary genus.

\begin{lem} \label{lem7.3}
Suppose $C$ is a general curve of genus $g \geq 7$. Then the only line bundles computing $\gamma_1$ 
of degree $\leq g-1$ have degree $d_1$, except when $g =9$, where there are also bundles of degree $d_2$ 
computing $\gamma_1$.
\end{lem}

\begin{proof}
For $g = 9$, we have $d_1 = 6$ and $d_2 = 8 = g-1$. The result follows.

For $g \neq 9$, we have to show that $d_r -2r > \gamma_1$ whenever $d_r \leq g-1$ and $r > 1$.

By \eqref{dr}, the condition $d_r \leq g-1$ is equivalent to $g \geq (r+1)^2$. So we require to prove that this implies that
$$
g-r - \left[ \frac{g}{r+1} \right] > \left[ \frac{g-1}{2} \right].
$$
It is sufficient to prove
$$
g-r - \frac{g}{r+1} > \frac{g-1}{2},
$$
which is equivalent to
$$
g(r-1) > (2r-1)(r+1).
$$
Since $g \geq (r+1)^2$, this is true for $r > 2$. For $r=2$ we have $d_2 = g+2 - \left[\frac{g}{3} \right]$. 
For $g =7$ and 8 this is $> g-1$. For $g \geq 10, \; d_2 > d_1 + 2$. So bundles of degree $d_2$ cannot compute $\gamma_1$.
\end{proof}

\begin{theorem} \label{prop7.5}
Let $C$ be a general curve of genus $g \geq 7, \;  g \neq 8$. Suppose that $\gamma'_2 = \gamma_1$. 
Then $\gamma_2 < \gamma'_2$ and
the bundles computing $\gamma'_2$ are
\begin{enumerate}
\item $Q \oplus Q'$ where $Q,\, Q'$ are bundles computing $\gamma_1$ of degree $d_1$ or, if $g = 9$, of degree $d_2$;
\item possibly non-trivial extensions
$$
0 \ra Q \ra E \ra K \otimes Q'^* \ra 0
$$
where all sections of $K \otimes Q'^*$ lift to $E$;
\item for $g$ odd, non-trivial extensions
$$
0 \ra Q \ra E \ra Q' \ra 0
$$
where all sections of $Q'$ lift to $E$ (such extensions always exist when $Q \simeq Q'$ and $d_Q=d_1$);
\item possibly stable bundles not possessing a line subbundle with $h^0 \geq 2$. We have necessarily
$h^0(E) = 2+s$ with
\begin{equation} \label{eq*}
2 \leq s \leq \left\{ \begin{array}{lll}
                      \frac{g-1}{2} & \mbox{if} & g \; \mbox{is odd},\\
                      \frac{g-2}{4} & \mbox{if} & g \; \mbox{is even}
                      \end{array} \right.
\end{equation}
and $d_E = 2 \gamma_1 + 2s$.
\end{enumerate}
\end{theorem}

\begin{proof}
We have $\gamma_2 < \gamma'_2$ by $(2.1)$ and $(7.1)$. Clearly the bundles of type $(1)$ compute $\gamma'_2=\gamma_1$.

If $E$ computes $\gamma'_2$ and $E$ has a subbundle $M$ with $h^0(M) \geq 2$, then by Lemma \ref{lempr2} we must 
have an extension $0 \ra M \ra E \ra N \ra 0$ with $\gamma(M) = \gamma(N) = \gamma_1$ and all sections of $N$ lift to $E$.
In view of Lemma \ref{lem7.3} the only possibilities are types (2) and (3). 

For $g \geq 10$ even, it follows from  Lemmas \ref{mult} and \ref{lem2.10} that there exists no non-trivial extension 
of type (3) with $Q \simeq Q'$. If $Q \not \simeq Q'$, it follows from \cite[Proposition 4.1]{v} that $h^0(Q \otimes Q') = 4$ 
and Lemma \ref{mult} applies again. For $g$ odd, the existence of non-trivial extensions when $Q \simeq Q'$ follows 
again from Lemmas \ref{mult} and \ref{lem2.10}. 

If $E$ computes $\gamma'_2$ and $E$ does not admit a subbundle with $h^0 \geq 2$, then $h^0(E) = 2+s$ with $s \geq 2,
\; d_E = 2\gamma_1 + 2s \leq 2g-2$ and $ d_E \geq d_{2s}$ by Lemma \ref{lempr}.
Any quotient line bundle $L$ of $E$ must have $h^0(L) \geq s+1$. So $d_L \geq d_s$. If $E$ is strictly semistable, then
$d_E \geq 2d_s$, giving $\gamma_1 + s \geq d_s$. Since $\gamma_1 = d_1 -2$, this contradicts the fact that $d_s \geq d_1 + s - 1$.
Hence $E$ is stable.

It remains to prove \eqref{eq*}.
By \eqref{dr} we have  
$$
2 \left[ \frac{g-1}{2} \right] + 2s \geq g + 2s - \left[ \frac{g}{2s+1} \right],
$$
which is equivalent to the second inequality of \eqref{eq*}.
\end{proof}

\begin{rem} \label{r7.5}
Semistable bundles of type (2) exist when $Q \simeq Q'$ and $g = 7, \;9$ or $11$ by dimensional calculations 
using Lemma \ref{l2.3}.
Such bundles do not exist when $Q \simeq Q'$ and $g \geq 10,\; g \neq 11$, by \cite[Propositions 4.2 and 4.3]{v}. 
\end{rem}

\begin{rem} \label{r7.6}
For all odd $g$, there exist bundles of type (4) with $s=2$ \cite[Theorem 1.1]{fo} (see also \cite[postscript]{gmn}). Provided that $\gamma_2'=\gamma_1$, these bundles compute $\gamma_2'$.
\end{rem}

\begin{prop} \label{prop7.6}
Let $C$ be a general curve of genus $7$. Bundles of type $(4)$ exist for $s=3$. 
\end{prop}

\begin{proof} 
If $s = 3$, we have $\gamma_1 = 3$ and $d_E = 12$.
By Lemma \ref{lempr}, $h^0(\det E) \geq 7$. So $\det E \simeq K$. By \cite{ms} the semistable $E$ possesses a 
subbundle $M$ of degree $\geq 3$. Considering cases, we see that we must have an extension
\begin{equation} \label{eq7.1}
0 \ra M \ra E \ra K \otimes M^* \ra 0
\end{equation}
with $h^0(M) = 1, \; d_M = 3, \; h^0(K \otimes M^*) = 4$ and all sections of $K \otimes M^*$ lift to $E$. By Lemma \ref{l2.3} 
and the fact that $h^0(K^2 \otimes {M^*}^2) = 12$ by Riemann-Roch, we see that there exists such an extension 
for every $M$.

Suppose that $E$ is not semistable. By considering cases, we see that there exists an extension
$$
0 \ra M' \ra E \ra K \otimes M'^* \ra 0
$$
with $d_{M'} = 7$ and $h^0(M') = 3$. Hence there exists a nonzero homomorphism $M \ra K \otimes M'^*$. 
Since $K \otimes M'^*$ is a generated line bundle of degree $5$ with $h^0(K \otimes M'^*) = 2$, there is at most
a 1-dimensional system of such line bundles $M$ for any fixed $M'$. Moreover, $M'$ belongs to the Brill-Noether 
locus of line bundles of degree 7 with $h^0 \geq 3$ which on a general curve of genus 7 has dimension 1.  
So the system of $M$ for which such an $M'$ exists has dimension at most 2. 
Hence for a general $M$ for which an extension \eqref{eq7.1} exists there is no such $M'$.
\end{proof}

\begin{prop} \label{prop7.7}
Let $C$ be a general curve of genus $9$. Then bundles of type $(4)$ with $s = 4$ exist.
\end{prop}

\begin{proof}
The proof is similar to the proof of Proposition \ref{prop7.6}. In this case, $d_M = 4, \; h^0(K \otimes M^*) = 5$ 
and $h^0(K^2 \otimes {M^*}^2) = 16$. So Lemma \ref{l2.3} applies again.
\end{proof}

\begin{rem}
If $g =9$ and $s=3$, we are not able to decide whether any bundles of type (4) exist. If $g$ is odd $\geq 11$, 
the argument of Propositions \ref{prop7.6} and \ref{prop7.7} no longer works, even in the case $s = \frac{g-1}{2}$.
\end{rem}

\begin{prop}\label{prop7.10}
Let $C$ be a general curve of genus $10$. Then $\gamma_2 < \gamma'_2 = \gamma_1$ and the only bundles computing $\gamma_2'$ are the bundles $Q\oplus Q'$, where $Q$ and $Q'$ are bundles of degree $6$ with $h^0=2$.
\end{prop}
\begin{proof} 
In this case, if $E$ is a bundle of type (4), we must have $s=2$ and $d_E = 12$. There do not exist
any stable bundles $E$ of this type by \cite[Theorem 4.1(i)]{gmn}. For bundles of type (2), it follows from \cite[postscript]{gmn} that
$$
H^0(K \otimes Q'^*) \otimes H^0(K \otimes Q^*) \ra H^0(K^2 \otimes Q^* \otimes Q'^*)
$$ is always surjective. Hence, by Lemma \ref{l2.3}, there are no extensions of type (2). This leaves only type (1).
\end{proof}

\section{Curves with $\gamma'_2 < \gamma_1$}

We have already noted that a general curve of genus $g \leq 16$ has $\gamma'_2 = \gamma_1$ and it is 
conjectured that this holds for general curves of arbitrary genus. However there are examples of curves of any genus $g \geq 11$ 
for which $\gamma'_2 < \gamma_1$ (see \cite{fo}).

In this section we shall refer to stable bundles $E$ with $h^0(E) = 2 + s, \; s \geq 2$, not possessing a line subbundle 
with $h^0 \geq 2$ as 
{\it bundles of type PR}.

We begin by considering bundles computing $\gamma_2$.

\begin{prop}\label{prop8.2}
Suppose $\gamma_2'<\gamma_1$. Then $\gamma_2=\frac{d_2}2-1$ and the bundles computing $\gamma_2$ but not $\gamma_2'$ are precisely the bundles $E_L$ where $L$ is a line bundle with $d_L=d_2$ and $h^0(L)=3$.
\end{prop}\begin{proof}
It follows from \eqref{eq4.1} that $\gamma_2=\frac{d_2}2-1<\gamma_1$. The result now follows from Corollary \ref{cor2}.
\end{proof}

We turn now to the consideration of bundles computing $\gamma_2'$.

\begin{prop} \label{prop8.1}
Suppose $\gamma'_2 < \gamma_1$. Then all bundles computing $\gamma'_2$ are of type PR with 
$$
2 \leq s \leq \gamma'_2 - \frac{\gamma_1}{2}.
$$
\end{prop}

\begin{proof}
If $E$ is a semistable bundle computing $\gamma'_2$ and possessing a line bundle $M$ with $h^0(M) \geq 2$, then by 
Lemma \ref{lempr2}, $\gamma(E) \geq \gamma_1$, a contradiction. So $E$ is of type PR
and 
$$
d_{2s} \leq d_E \leq 2g-2.
$$
Since $d_{g-1} = 2g-2$, this implies that $2s \leq g-1$. By \eqref{eq0} we have 
$$
d_{2s} \geq \min \{ \gamma_1 +4s, g+2s -1 \}.
$$
If $g-1- \gamma_1 \leq 2s \leq g-1$, this gives $d_{2s} \geq g+2s -1$, but 
$$
d_{2s} \leq g+2s - \left[ \frac{g}{2s+1} \right] = g+2s - 1.
$$ So $d_{2s} = g+2s -1$ and $\gamma'_2 = \gamma(E) \geq \frac{g-1}{2}$,
a contradiction.

Hence $2s < g-1-\gamma_1$ and $d_{2s} \geq \gamma_1 + 4s$ which implies 
$$
\gamma'_2 = \gamma(E) \geq \frac{\gamma_1 +2s}{2} = \frac{\gamma_1}{2} +s.
$$
So $s \leq \gamma'_2 - \frac{\gamma_1}{2}$.
Stability of $E$ follows as in the proof of Theorem \ref{prop7.5}, since $d_1 - 2 \geq \gamma_1$.
\end{proof} 

\begin{theorem} \label{thm8.2}
Suppose $\gamma'_2 < \gamma_1$ and $d_4 = 2 \gamma'_2 + 4$. Then the set of bundles 
of type PR with $s=2$ which compute $\gamma'_2$ is in bijective correspondence with the set of line bundles
$$
U(d_4,5) := \left\{ M \;{\Big |}\; \begin{array}{c}
                      d_M = d_4,\quad  h^0(M) = 5,\\
                      S^2H^0(M) \ra H^0(M^2) \; \mbox{not injective}
                      \end{array} \right\}.
$$
\end{theorem}

\begin{proof}
If $E$ is a bundle of type PR with $s =2$ which computes $\gamma'_2$, then
$$
d_E = 2 \gamma'_2 + 4 = d_4 < 2d_1.
$$
Since $E$ is necessarily stable, the result follows from
\cite[Theorem 3.2 and Remark 3.4]{gmn}.
\end{proof}

The following corollary generalises \cite[Proposition 4.5]{lmn}
\begin{cor} \label{cor8.3}
Suppose $\gamma_1\ge5$ and $\gamma'_2 = \frac{\gamma_1}{2} + 2$. Then $U(d_4,5)$ is non-empty and the corresponding bundles of type PR 
are the only bundles that compute $\gamma'_2$.
\end{cor}

\begin{proof}
By Proposition \ref{prop8.1}, every bundle $E$ computing $\gamma_2'$ is of type PR with $s=2$. By \eqref{eq4.1}, $\gamma'_2 \geq \frac{d_4}{2} - 2$. So $d_4 \leq 2 \gamma'_2 + 4$.

On the other hand, $d_4 \geq \gamma_1 + 8$ by \eqref{eq0}. So $d_4 \geq 2 \gamma'_2 + 4$ and hence
$$
d_4 = 2 \gamma'_2 + 4.
$$
The result follows from the theorem. 
\end{proof}

\begin{cor} \label{cor8.4}
Suppose $\gamma'_2 < \gamma_1 =5$. Then $U(d_4,5)$ is non-empty and the corresponding bundles of type PR are 
the only bundles that compute $\gamma'_2$. 
\end{cor}

\begin{proof}
In this case $\gamma'_2 = \frac{9}{2} = \frac{\gamma_1}{2} + 2$, since $\gamma'_2 < \gamma_1 = 5$ and $\gamma'_2 \geq \frac{9}{2}$
by \eqref{eq4.1}. So the assertion follows from the previous corollary.
\end{proof}

\begin{prop} \label{prop8.5}
Suppose $\gamma_1 \geq 6$ and $\gamma'_2 = \frac{\gamma_1 + 5}{2}$. Then the bundles computing $\gamma'_2$ are all of type PR 
with $s = 2$ and $d_E = 2 \gamma'_2 + 4$. The set of such bundles is in bijective correspondence with 
$$
U^0(2\gamma'_2 + 4,5) := \{ M \in U(2\gamma'_2 + 4,5) \;|\; M\ generated \}.
$$
\end{prop}

\begin{proof}
We have $s=2$ by Proposition \ref{prop8.1}.
By \eqref{eq0} we have $d_5 \geq \gamma_1 + 10 > 2 \gamma'_2 + 4$.
Then the argument of \cite[Theorem 3.2 and Remark 3.4]{gmn} gives the assertion. 
\end{proof}

\begin{cor} \label{cor8.6}
Suppose $\gamma'_2 < \gamma_1 = 6$. Then $\gamma_2'=5$ or $\frac{11}2$ and  $d_4 = 14$ or $15$.
\begin{enumerate}
\item If $\gamma'_2 = 5$, then the set of bundles computing $\gamma'_2$ is in bijective correspondence with $U(14,5)$.
\item If $\gamma'_2 = \frac{11}{2}$, then the set of bundles computing $\gamma'_2$ is in bijective correspondence with $U^0(15,5)$.
\end{enumerate}
If $d_4 = 15$, then $\gamma'_2 = \frac{11}{2}$ and $U^0(15,5)=U(15,5)$.
\end{cor}

\begin{proof}
Since $\gamma'_2 \geq \frac{\gamma_1}{2} + 2$, we must have either $\gamma'_2 = 5$ or $\gamma'_2 = \frac{11}{2}$.
Since $\gamma_1 + 8 \leq d_4 \leq 2\gamma'_2 + 4$, we have $d_4 = 14$ or $15$. 

If $\gamma'_2 = 5$, the result follows from Theorem \ref{thm8.2}. If $\gamma'_2 = \frac{11}{2}$, it follows from Proposition
\ref{prop8.5}.
\end{proof}

\begin{ex}
We know from \cite{fo} that there exist examples of curves of genus 11 and 12 with $\gamma_1 = 5$ and $\gamma'_2 = \frac{9}{2}$. 
So Corollary \ref{cor8.4} applies.
\end{ex}

\begin{ex}
From the same source we know that there exist curves of genus 13 with $\gamma_1 = 6$ and $\gamma'_2 \leq \frac{11}{2}$. 
So Corollary \ref{cor8.6} applies. We are not certain whether there exist curves of genus 13 with $\gamma_1 = 6$ and 
$\gamma'_2 =5$. 
\end{ex}

\begin{ex}
By \cite[Theorem 1.1]{ln2} there exists  a curve of genus 14 with $\gamma_1 = 6$ and $\gamma'_2 = 5$ provided 
that the quadratic form
$$
3m^2 + 14 mn + 13 n^2
$$
cannot take the value $-1$ for any integers $m$ and $n$. 

If the quadratic form does take the value $-1$, then reduction modulo 4 shows that $m$ is odd and $n$ is even. 
Writing $m = 2m_1 + 1$ and $n = 2n_1$ reduces the equation to
$$
3(m_1^2 + m_1) + 7n_1(2m_1 + 1) + 13n_1^2 = -1.
$$
The left hand side of this equation is always even, a contradiction. So the curve exists.
\end{ex}

\end{document}